\documentclass[a4paper,reqno]{amsart}

\usepackage{graphicx}
\usepackage{mathrsfs}
\usepackage{amsfonts}
\usepackage{amssymb}
\usepackage{amsmath}
\usepackage{amsthm}
\usepackage{amscd}
\usepackage[all,2cell]{xy}
\usepackage{color}
\usepackage[pagebackref,colorlinks]{hyperref}
\usepackage{enumerate}
\usepackage{smartdiagram}

\usepackage{tikz}
\usetikzlibrary{shapes,shadows,arrows}
\usepackage{tkz-graph}
\usetikzlibrary{arrows,shapes,positioning}

\usepackage{comment}
\UseAllTwocells \SilentMatrices

\numberwithin{equation}{section}

\theoremstyle{plain}
\newtheorem{thm}{Theorem}[section]
\newtheorem{prop}[thm]{Proposition}
\newtheorem{cor}[thm]{Corollary}
\newtheorem{lem}[thm]{Lemma}
\newtheorem{conjecture}[thm]{Conjecture}

\theoremstyle{definition}

\newtheorem{example}[thm]{Example}
\theoremstyle{remark}
\newtheorem{rmk}[thm]{\bf Remark}

\newcommand{\Modd}{\operatorname{Mod-\!}}
\newcommand{\Grr}{\operatorname{Gr-\!}}

\def\xra{\xrightarrow[]{}}
\def\a{\alpha}

\def\v{\varepsilon}

\def\VV{\mathcal{V}}
\def \Z{\mathbb Z}
\def\-{\text{-}}
\def\TT{\mathcal{T}}
\def\LL{\mathcal{L}}

\newcommand{\gr}{\operatorname{gr}}

\def \Mn {M^{\gr}_E}
\def \Mnb {M_{\overline{E}}}

\begin{document}

\title{The talented monoid of a Leavitt path algebra}

\author{Roozbeh Hazrat}
\address{Roozbeh Hazrat: 
Centre for Research in Mathematics\\
Western Sydney University\\
Australia} \email{r.hazrat@westernsydney.edu.au}

\author{Huanhuan Li}
\address{
Huanhuan Li: Centre for Research in Mathematics\\
Western Sydney University\\
Australia} \email{h.li@westernsydney.edu.au}

\subjclass[2010]{18B40,16D25}

\keywords{Leavitt path algebra, graded Grothendieck group, graded ring, graph monoid}

\date{\today}

\begin{abstract} 

There is a tight relation between the geometry of a directed graph and the algebraic structure of a Leavitt path algebra associated to it. In this note, we show a similar connection between the geometry of the graph and the structure of a certain monoid associated to it. This monoid is isomorphic to the positive cone of the graded $K_0$-group of the Leavitt path algebra which is naturally equipped with a $\mathbb Z$-action. As an example, we show that a graph has a cycle without an exit if and only if the monoid has a periodic element. Consequently a graph has Condition (L) if and only if the group $\mathbb Z$ acts freely on the monoid. We go on to show that the algebraic structure of Leavitt path algebras (such as simplicity, purely infinite simplicity, or the lattice of ideals) can be described completely via this monoid. Therefore an isomorphism between the monoids (or graded $K_0$'s) of two Leavitt path algebras implies that the algebras have similar algebraic structures.  
These all confirm that the graded Grothendieck group could be a sought-after complete invariant for the classification of Leavitt path algebras. 
\end{abstract}

\maketitle

\section{Introduction}

The theory of Leavitt path algebras has sparked a substantial amount of activity in recent years culminating in finding, on the one hand, a complete algebraic structure of these algebras via the geometry of the associated graphs and in finding, on the other hand, a complete invariant for the classification of the algebras. The first two papers in the subject appeared in 2005 and 2006 \cite{abrams2005, ara2006}. The first paper \cite{abrams2005} gave a graph criteria when these algebras are simple and the second paper \cite{ara2006} proved that the non-stable $K$-theory of these algebras can be described via a natural monoid associated to their graphs. In this paper we tie these two threads together by showing  how the geometry of a graph is closely related to the structure of \emph{graded} monoid of the associated Leavitt path algebra.

Let $E$ be a (row-finite) directed graph, with vertices denoted by $E^0$ and edges by $E^1$. The monoid $M_E$ considered in \cite{ara2006} is defined as the free abelian monoid over the vertices  subject to identifying a vertex with the sum of vertices it arrives at by the edges emitting from it: 
\begin{equation*}
M_E= \Big \langle \, v \in E^0 \, \, \Big \vert \, \,  v= \sum_{v\rightarrow u} u \, \Big \rangle.
\end{equation*}

It was proved in \cite{ara2006}, using Bergman's machinery, that $M_E = \mathcal V (L_F(E))$. Here  $\mathcal V(L_F(E))$ is the monoid of finitely generated projective modules  of the Leavitt path algebra $L_F(E)$, with coefficients in a field $F$. Thus the group completion of $M_E$ retrieves the Grothendieck group $K_0(L_F(E))$. For half a century, this group has played a key role in the classification of  $C^*$-algebras and in particular graph $C^*$-algebras which are the analytic counterpart of Leavitt path algebras~\cite{tomforde1}. 

The ``graded'' version of this monoid is defined as  
\begin{equation*}
\Mn= \Big \langle \, v(i), v \in E^0, i \in \mathbb Z  \, \,  \Big \vert \, \, v(i)= \sum_{v\rightarrow u} u(i+1) \, \Big \rangle.
\end{equation*}
Note that the only difference from the monoid $M_E$ is that we index the vertices by $\mathbb Z$ and keep track of the transformations. There is a natural action of $\mathbb Z$ on $\Mn$:  the action of $n\in \mathbb Z$ on $v$ is defined by  $v(n)$ and denoted by ${}^n v$. It was proved 
in \cite{ahls} that $\Mn = \mathcal V^{\gr} (L_F(E))$ (see also Remark \ref{rmk}). Here  $ \mathcal V^{\gr} (L_F(E))$ is the monoid of graded finitely generated projective modules of the Leavitt path algebra $L_F(E)$. Thus the group completion of $\Mn$ is the graded Grothendieck group $K^{\gr}_0(L_F(E))$. The action of $\mathbb Z$ on $\Mn$ corresponds to the shift operation on graded modules over the Leavitt path algebra $L_F(E)$ which is naturally a $\mathbb Z$-graded ring. 

The aim of this note is to show that there is a beautiful and close relation between the geometry of a graph $E$ and the monoid structure of $\Mn$ parallel to the correspondence between the algebraic structure of $L_F(E)$ and the geometry of the graph $E$, as the figure below indicates. 

\begin{center}
\tikzstyle{decision} = [diamond, draw, fill=red!50]
\tikzstyle{line} = [draw,  -stealth, thick]
\tikzstyle{lined} = [draw, bend right=45, thick, dashed]

\tikzstyle{elli}=[draw, ellipse, top color=white, bottom color=white ,minimum height=8mm, text width=6.3em, text centered]
\tikzstyle{elli2}=[draw, ellipse ,minimum height=8mm, text width=6.3em, text centered]
\tikzstyle{elli3}=[draw, ellipse, top color=white, bottom color=blue!90 ,minimum height=8mm, text width=6.3em, text centered]

\tikzstyle{block} = [draw, rounded corners, rectangle, text width=8em, text centered, minimum height=15mm, node distance=7em]
\tikzstyle{block2} = [draw, rounded corners, rectangle, top color=white!80, bottom color=white, text width=8em, text centered,  minimum height=15mm, node distance=7em]


\begin{tikzpicture}[scale=0.7, transform shape]

\GraphInit[vstyle = Shade]

\node[block] (graph) {\bf  Geometry of the graph $\pmb E$};
\node[elli2, below of=graph, xshift=-12em, yshift=-3em] (ring) { \bf Algebraic structure of $\pmb {L_F(E)}$};
\node[elli2, below of=graph, xshift=12em, yshift=-3em] (monoid) {\bf Monoid structure of $\pmb \Mn$};

\tikzset{
  EdgeStyle/.append style = {<->, thin, dotted} 
  }
 \Edge (ring)(monoid)

\tikzset{
  EdgeStyle/.append style = {<->, bend right, dashed} 
  }
\Edge (graph)(ring)

\tikzset{
  EdgeStyle/.append style = {<->, bend left} 
  }
\Edge (graph)(monoid)

\end{tikzpicture}


\end{center}

In turn this shows that if the graded monoids of two Leavitt path algebras are isomorphic, then the algebraic properties of one algebra induces the same properties in the other algebra via the graded monoid $M^{\gr}$ as a bridge.

Any monoid is equipped with a pre-ordering; $a \leq b$ if $b=a+c$.  We will see that the pre-order structure of $\Mn$ of the graph $E$ determines the graded structure of the Leavitt path algebra $L_F(E)$, whereas the action of the group $\mathbb Z$ gives information about the non-graded structure.  
  We show that the properties of a graph $E$  having cycles with/without exits, can be translated as properties of the orbits of the action of the group $\mathbb Z$ on $\Mn$.  Specifically, we prove that a  graph $E$ has a cycle without an exit if and only if there is an element $a\in \Mn$  such that ${}^na  = a$,  for some $n\in \mathbb Z$  (Proposition~\ref{goldenprop}). Consequently the graph has condition (L) if and only if $\mathbb Z$ acts freely on $\Mn$ (Corollary~\ref{conLm}).
  We go further to show that the (non-graded) algebraic structure of $L_F(E)$ (such as simplicity or purely simplicity) can be described completely by the orbits of $\mathbb Z$-action on $\Mn$ (Corollary~\ref{conKm}) . We conclude the paper by proving that a $\mathbb Z$-module isomorphism between the monoids preserves important structures of corresponding Leavitt path algebras (Theorem~\ref{mainthemethe}).

The fact that the algebraic structure of the Leavitt path algebra $L_F(E)$ can be read via the monoid $\Mn$ should not come as a surprise if $\Mn$ were to be a complete invariant for these algebras.  In fact it was conjectured in \cite[Conjecture~1]{roozbehhazrat2013}, that  the graded Grothendieck group $K_0^{\gr}$ along with its ordering and its module structure is a complete invariant for the class of (finite) Leavitt path algebras (see also~\cite{arapardo}, \cite[\S~7.3.4]{AAS}).   
  
\begin{conjecture}\label{conj1}
Let $E_1$ and $E_2$ be finite graphs and $F$ a field. Then the following are equivalent.

\begin{enumerate}[\upshape(1)]

\item There is a $\mathbb Z$-module isomorphism $\phi: M^{\gr}_{E_1} \rightarrow  M^{\gr}_{E_2}$, such that $\phi\big (\sum_{v\in E_1^0} v\big)=\sum_{v\in E_2^0} v$;

\smallskip

\item There is an order-preserving $\Z[x,x^{-1}]$-module isomorphism 
\begin{align*}
K_0^{\gr}(L_F(E_1)) &\longrightarrow K_0^{\gr}(L_F(E_2)),\\
 [L_F(E_1)] &\longmapsto [L_F(E_2)]. 
 \end{align*}

\medskip 

\item There is a graded ring isomorphism $\varphi:L_F(E_1) \rightarrow L_F(E_2)$.
\end{enumerate}
\end{conjecture}  

Note that since $K^{\gr}_0(L_F(E))$ is the group completion of $\Mn$, the directions 1 $\Leftrightarrow$ 2 are immediate.


\section{Graph monoids}
 
 In this section we briefly introduce the notions of the directed graph and the monoid $M_E$ associated to it. We then introduce the ``graded'' version of this monoid $\Mn$. 
 We refer the reader to the recent monograph \cite{AAS} for the theory of Leavitt path algebras and a comprehensive study of the monoid $M_E$. The monoid $\Mn$ was first considered in \cite{roozbehhazrat2013} (see also \cite[\S3.9.2]{hazi}) as the positive cone of the graded Grothendieck group $K^{\gr}_0(L_F(E))$ and further studied in~\cite{haz3,ahls}.

\subsection{Graphs}\label{graphsec}

A directed graph $E$ is a tuple $(E^{0}, E^{1}, r, s)$, where $E^{0}$ and $E^{1}$ are
sets and $r,s$ are maps from $E^1$ to $E^0$. A graph $E$ is finite if $E^0$ and $E^1$ are both finite. We think of each $e \in E^1$ as an edge 
pointing from $s(e)$ to $r(e)$. A vertex $v\in E^0$ is a sink if $s^{-1}(v)=\emptyset$. We use the convention that a (finite) path $p$ in $E$ is
a sequence $p=\a_{1}\a_{2}\cdots \a_{n}$ of edges $\a_{i}$ in $E$ such that
$r(\a_{i})=s(\a_{i+1})$ for $1\leq i\leq n-1$. We define $s(p) = s(\a_{1})$, and $r(p) =
r(\a_{n})$. 

 If there is a path from a vertex $u$ to a vertex $v$, we write $u\ge v$. A subset $M$ of $E^0$ is \emph{downward directed}  if for any two $u,v\in M$ there exists $w\in M$ such that $u\geq w$ and $v\geq w$ (\cite[\S4.2]{AAS}, \cite[\S2]{rangaswamy}).
We will use the following two notations throughout: for $v\in E^0$,
\[T(v):=\{w\in E^0 \mid v\geq w \} \,   \text{ and } \, M(v):=\{w\in E^0 \mid w \geq v \}.\]

A graph $E$ is said to be \emph{row-finite} if for each vertex $u\in E^{0}$,
there are at most finitely many edges in $s^{-1}(u)$. A vertex $u$ for which $s^{-1}(u)$
is empty is called a \emph{sink}, whereas $u\in E^{0}$ is called an \emph{infinite
emitter} if $s^{-1}(u)$ is infinite. If $u\in E^{0}$ is neither a sink nor an infinite
emitter, then it is called a \emph{regular vertex}. 

Following now the standard notations (see \cite[\S 2.9]{AAS}), we denote by $E^\infty$ the set of all infinite paths and by $E^{\leq \infty}$ the set $E^\infty$ together with the set of finite paths in $E$ whose range vertex is a singular vertex. For a vertex $v\in E^0$, we denote by $vE^{\leq \infty}$ the paths in $E^{\leq \infty}$ that starts from the vertex $v$. 

We define the ``local'' version of cofinality which we will use in \S\ref{seciv}. We say a vertex $v\in E^0$ is \emph{cofinal with respect} to $w\in E^0$ if for every $\alpha \in wE^{\leq \infty}$, there is a path from $v$ which connects to a vertex in $\alpha$.   We say a vertex $v$ is cofinal, if it is cofinal with respect to any other vertex. Finally, we say the graph $E$ is \emph{cofinal} if every vertex is cofinal with respect to any other vertex. The concept of cofinal graph was originally used to give a criteria for the simplicity of graph algebras (see \cite[Theorem~2.9.7]{AAS} and \cite[\S5.6]{AAS}). 

We say a vertex $v\in E^0$ has \emph{no bifurcation} if for any $u\in T(v)$, $| s^{-1}(u)| \leq 1$. We say $v$ is a \emph{line-point} if $v$ has no bifurcation and does not end at a cycle. Note that by our definition, a sink is a line-point (see \cite[\S 2.6]{AAS}). 

Throughout the note the graphs we consider are row-finite graphs. The reason is twofold. The proofs are rather more transparent in this case and the conjecture that the graded Grothendieck group for Leavitt path algebras is a complete invariant (Conjecture~\ref{conj1}) is originally formulated for finite graphs. We think one could generalise the results of the paper to arbitrary graphs and possibly other Leavitt path-like algebras, such as weighted Leavitt path algebras.

Throughout the paper we will heavily use the concept of the covering of a graph. The \emph{covering graph}  $\overline E$  of $E$ (also denoted by  $E\times_1 \mathbb Z$) is given by
\begin{gather*}
    \overline E^0 = \big\{v_n \mid v \in E^0 \text{ and } n \in \Z \big\},\qquad
    \overline E^1 = \big\{e_n \mid e\in E^1 \text{ and } n\in \Z \big\},\\
    s(e_n) = s(e)_n,\qquad\text{ and } \qquad  r(e_n) = r(e)_{n+1}.
\end{gather*}

As examples, consider the following graphs
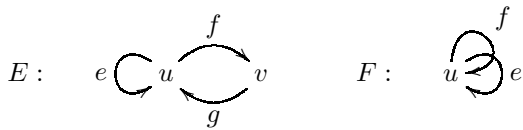
\begin{equation*}
{\def\labelstyle{\displaystyle}
E : \quad \,\, \xymatrix{
 u \ar@(lu,ld)_e\ar@/^0.9pc/[r]^f & v \ar@/^0.9pc/[l]^g
 }} \qquad \quad
{\def\labelstyle{\displaystyle}
F: \quad \,\, \xymatrix{
   u \ar@(ur,rd)^e  \ar@(u,r)^f
}}
\end{equation*}
Then
\begin{equation}\label{level421}
\xymatrix@=15pt{
& \text{\bf Level -1} && \text{\bf Level 0} && \text{\bf Level 1} && \text{\bf Level 2}\\ 
\overline E: & \dots  {u_{-1}} \ar[rr]^-{e_{-1}} \ar[drr]^(0.4){f_{-1}} &&  {u_{0}} \ar[rr]^-{e_0} \ar[drr]^(0.4){f_0} && {u_{1}}  \ar[rr]^-{e_{1}} \ar[drr]^(0.4){f_{1}} && \cdots\\
& \dots {v_{-1}}   \ar[urr]_(0.3){g_{-1}} && {v_{0}} \ar[urr]_(0.3){g_0}  && {v_{1}}  \ar[urr]_(0.3){g_{1}}&& \cdots
}
\end{equation}
and
\begin{equation}\label{level422}
\xymatrix@=15pt{
\overline{F}: \quad \,\,&\dots  {u_{-1}} \ar@/^0.9pc/[rr]^{f_{-1}} \ar@/_0.9pc/[rr]_{e_{-1}}  &&  {u_{0}} \ar@/^0.9pc/[rr]^{f_0} \ar@/_0.9pc/[rr]_{e_0} && {u_{1}}  \ar@/^0.9pc/[rr]^{f_{1}}  \ar@/_0.9pc/[rr]_{e_{1}} && \quad \cdots
}
\end{equation}

Notice that for any graph $E$, the covering graph $\overline E$ is an acyclic \emph{stationary} graph, meaning, the graph repeats the pattern going from ``level'' $n$ to ``level'' $n+1$. This fact will be used throughout the article.

Recall that a subset $H \subseteq E^0$ is said to be \emph{hereditary} if
for any $e \in E^1$ we have that $s(e)\in H$ implies $r(e)\in H$. A hereditary subset $H
\subseteq E^0$ is called \emph{saturated} if whenever $0 < |s^{-1}(v)| < \infty$, then $\{r(e)\mid
e\in E^1 \text{~and~} s(e)=v\}\subseteq H$ implies $v\in H$. Throughout the paper we work with hereditary saturated subsets of $E^0$.

For hereditary saturated subsets $H_1$ and $H_2$ of $E$ with  $H_1 \subseteq  H_2$, define the quotient graph $H_2 / H_1 $ as a graph such that 
$(H_2/ H_1)^0=H_2\setminus H_1$ and $(H_2/H_1)^1=\{e\in E^1\;|\; s(e)\in H_2, r(e)\notin H_1\}$. The source and range maps of $H_2/H_1$ are restricted from the graph $E$. If $H_2=E^0$, then $H_2/H_1$ is the \emph{quotient graph} $E/H_1$ (\cite[Definition~2.4.11]{AAS}).

\subsection{Graph Monoids} \label{monsec}
Let $M$ be an abelian monoid with a group $\Gamma$ acting on it.  For $\alpha \in \Gamma$ and $a\in M$, we denote the action of $\alpha$ on $a$ by ${}^\alpha a$. 
A monoid homomorphism $\phi:M_1 \rightarrow M_2$ is called $\Gamma$-\emph{module homomorphism} if $\phi$ respects the action of $\Gamma$, i.e., $
\phi({}^\alpha a)={}^\alpha \phi(a)$. We define a  pre-ordering on the monoid $M$ by $a\leq b$ if $b=a+c$, for some $c\in M$. 
Throughout we write $a \parallel b$ if the elements $a$ and $b$ are not comparable.  An element $a\in M$ is called an \emph{atom} if $a=b+c$ then $b=0$ or $c=0$. An element $a\in M$ is called \emph{minimal} if $b\leq a$ then $a\leq b$. When $M$ is cofinal and cancellative, these notions coincide with the more intuitive definition of minimality, i.e., $a$ is minimal if $0\not = b\leq a$ then $a=b$. The monoid of interest in the paper, $\Mn$, is cofinal and cancellative and thus all these concepts coincide. 

Throughout we assume that the group $\Gamma$ is abelian. Indeed in our setting of graph algebras, this group is the group of integers $\mathbb Z$. We use the following terminologies throughout this paper: We call an element $a\in M$ \emph{periodic} if there is an $\alpha \in \Gamma$ such that ${}^\alpha a =a$. If $a\in M$ is not periodic, we call it \emph{aperiodic}. We denote the orbit of the action of $\Gamma$ on an element $a$ by $O(a)$, so $O(a)=\{{}^\alpha a \mid \alpha \in \Gamma \}$.

 A $\Gamma$-\emph{order-ideal} of a monoid $M$ is a  subset $I$ of $M$ such that for any $\alpha,\beta \in \Gamma$, ${}^\alpha a+{}^\beta b \in I$ if and only if 
$a,b \in I$. Equivalently, a $\Gamma$-order-ideal is a submonoid $I$ of $M$ which is closed under the action of $\Gamma$ and it  is
\emph{hereditary} in the sense that $a \le b$ and $b \in I$ imply $a \in I$. The set $\LL(M)$ of $\Gamma$-order-ideals of $M$ forms a (complete) lattice. We say $M$ is \emph{simple} if the only $\Gamma$-order-ideals of $M$ are $0$ and $M$.

For a ring $A$ with unit, the isomorphism classes of finitely generated projective (left/right) $A$-modules with direct sum as an operation form a monoid denoted by $\mathcal V(A)$. 
This construction can be extended to non-unital rings via idempotents. For a $\Gamma$-graded ring $A$, considering the graded finitely generated projective modules, it provides us with the monoid  $\mathcal V^{\gr}(A)$ which has an action of $\Gamma$ on it via the shift operation on modules (see \cite[\S 3]{hazi} for the general theory).

In this article we consider these monoids when the algebra is a Leavitt path algebra. 
Ara, Moreno and Pardo \cite{ara2006} showed that for a Leavitt path algebra associated to a row-finite
graph $E$, the monoid $\mathcal{V}(L_{F}(E))$ is entirely determined by elementary
graph-theoretic data.   Specifically, for a row-finite graph $E$, we define $M_E$ to be the
abelian monoid generated by $E^{0}$ subject to
\begin{equation}\label{monoidrelation}
v=\sum_{e\in s^{-1}(v)}r(e), 
\end{equation}
for every $v\in E^{0}$ that is not a sink. Theorem~3.5 of~\cite{ara2006} relates this monoid to the theory of Leavitt path algebras:
There is a monoid isomorphism  $\mathcal{V}(L_{F}(E)) \cong M_E$.

There is an explicit description of the congruence on the free abelian
monoid given by the defining relations of $M_{E}$ \cite[\S 4]{ara2006}. Let $F_E$ be the free abelian monoid on
the set $E^{0}$. The nonzero elements of $F$ can be written in a unique form up to
permutation as $\sum_{i=1}^{n}v_{i}$, where $v_{i}\in E^{0}$. Define a binary relation
$\xra_{1}$ on $F\setminus\{0\}$ by 
\begin{equation}\label{hfgtrgt655}
\sum_{i=1}^{n}v_{i}\longrightarrow_{1}\sum_{i\neq
j}v_{i}+\sum_{e\in s^{-1}(v_{j})}r(e),
\end{equation}
whenever $j\in \{1, \cdots, n\}$ and 
$v_{j}$ is not a sink. Let $\xra$ be the transitive and reflexive closure of $\xra_{1}$
on $F\setminus\{0\}$ and $\sim$ the congruence on $F$ generated by the relation $\xra$.
Then $M_{E}=F/\sim$.

The following two part lemma is crucial for our work and we frequently use it throughout the article. For the proof see \cite{ara2006} and \cite[\S 3.6]{AAS}

\begin{lem}\label{aralem6}
Let $E$ be a row-finite graph, $F_E$ the free abelian monoid generated by $E^0$ and $M_E$ the graph monoid of $E$. 

\begin{enumerate}[\upshape(i)]

\item If   $a=a_1+a_2$ and $a\rightarrow b$, where $a, a_1, a_2, b \in F_E \backslash \{ 0 \}$,  then $b$ can be written as $b=b_1+b_2$ with $a_1\rightarrow b_1$  and $a_2\rightarrow b_2$.

\medskip

\item (The Confluence Lemma) For $a, b \in F_E \backslash \{ 0 \}$, we have   $a=b$ in $M_E$ if and only if there is $c \in F_E \backslash  \{0 \}$ such that 
$a \rightarrow c$ and $b\rightarrow c$. 

\end{enumerate}
\end{lem}

For a row-finite graph $E$, we define the ``graded'' version of the monoid $M_E$, and denote it by $\Mn$, to be the
abelian monoid generated by $\{v(i) \mid v\in E^0, i\in \mathbb Z\}$ subject to
\begin{equation}\label{monoidrelation2}
v(i)=\sum_{e\in s^{-1}(v)}r(e)(i+1), 
\end{equation}
for every $v\in E^{0}$ that is not a sink and $i \in \mathbb Z$. The monoid $\Mn$ is equipped by a natural $\mathbb Z$-action: $${}^n v(k)=v(k+n)$$ for $n,k \in \mathbb Z$. Proposition~5.7 of \cite{ahls} relates this monoid to the theory of Leavitt path algebras: there is a  $\mathbb Z$-module isomorphism 
$\Mn \cong \mathcal{V}^{\gr}(L_{F}(E))$.  In fact we have 
\begin{align}\label{yhoperagen}
\Mn &\cong\mathcal V(L_F(\overline E))\cong\, \mathcal V^{\gr}(L_F(E)),
 \notag
\end{align} (see also Remark \ref{rmk}). Thus the monoid $\Mn$  is cofinal and cancellative (\cite[\S5]{ahls}). We will use these facts throughout this paper.

Throughout the article, we simultaneously use $v\in E^0$ as a vertex of $E$, as an element of $L_F(E)$ and the element $v=v(0)$ in $\Mn$, as the meaning will be clear from the context. For a subset $H\subseteq E^0$, the ideal it generates in $L_F(E)$ is denoted by $I(H)$, whereas the $\mathbb Z$-order-ideal it generates in $\Mn$ is denoted by $\langle H \rangle $.

Let $I$ be submonoid of the monoid $M$. Define an equivalence relation $\sim_{I}$ on $M$ as follows: For $a, b\in M$,  $a\sim_{I} b$ if there exist $i,j\in I$ such that  $a+i=b+j$ in $M$. The quotient monoid $M/I$ is defined as $M/\sim$. Observe that $a\sim_{I} 0$ in $M$ for any $a\in I$. If $I$ is an order-ideal then $a\sim_I 0$ if and only if $a\in I$. 

There is a natural relationship between the quotient monoids of $\Mn$ and quotient graphs of $E$ as the following lemma shows.

\begin{lem} \label{qmiso} Let $E$ be a row-finite graph. Suppose that $H_1\subseteq H_2$ with $H_1$ and $H_2$ two hereditary saturated subsets of $E^0$.  We have a $\mathbb Z$-module isomorphism of monoids 
\[M_{H_2/H_1}^{\gr}\cong M_{H_2}^{\gr}/M_{H_1}^{\gr}.\]
In particular, for a hereditary saturated subset $H$ of $E^0$, and the order-ideal $I=\langle H \rangle \subseteq \Mn$, we have a $\mathbb Z$-module isomorphism 
\[M_{E/H}^{\gr}\cong M_{E}^{\gr}/I.\]
\end{lem}
\begin{proof}
One can establish this directly and it is similar to the non-graded version which has already been established in the literature (see \cite[Ptoposition~3.6.18]{AAS}). 
\end{proof}

\begin{example}
Consider the graph $E$, 
\[
\xymatrix{
E : &o \ar@/^0.9pc/[r]^{\alpha} &  u  \ar@/^0.9pc/[l]^{\beta}  \ar@/^1.9pc/[rrr]^{\nu} \ar@/_0.9pc/[r]_{\gamma} & v   \ar@/^0.9pc/[rr]^{\mu} && x  \ar@/^0.9pc/[ll]^{\delta}}
\]
We do some calculation in the monoid $\Mn$ which prepares us for the theorems in the next section. 

Using the relations~\ref{monoidrelation2} in $\Mn$,  
$o=u(1)=o(2)+v(2)+x(2)$. Thus we have 
\[{}^{-2}  o >o.\] 
This follows because the cycle $\alpha\beta$ in $E$ has an exit. In fact we show in Proposition~\ref{goldenprop} that if there is an $a\in \Mn$ such that ${}^n a > a$ with $n$ a negative integer then there is a cycle with an exit in the graph. On the other hand, we have $v=x(1)=v(2)$, i.e.,  ${}^2  v =v.$ This is because the cycle $\mu\delta$ has no exit. We further prove in Proposition~\ref{goldenprop} that if there is an $a\in \Mn$ with ${}^n a=a$ then there is a cycle in $E$ without an exit. 

Consider now the hereditary saturated subset $\{v,x\}$ and consider the quotient graph $E/ \{v,x\}$ which is 
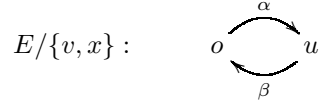
\[
\xymatrix{
E/ \{v,x\} : &o \ar@/^0.9pc/[r]^{\alpha} &  u  \ar@/^0.9pc/[l]^{\beta}}
\]
Thus in $M^{\gr}_{E / \{v,x\}} \cong M_E^{\gr} / \langle u, x \rangle$ we have 
 \[{}^{-2}  o  =o.\]
This follows because the cycle $\alpha\beta$ has all its exits in the hereditary saturated subset $\{v,x\}$. By the theory of Leavitt path algebras this gives non-graded ideals in $L_F(E)$. In Proposition~\ref{propnongra} we show that if for an order-ideal $I$, the quotient $\Mn/I$ has a periodic element, then $L_F(E)$ has non-graded ideals. 
\end{example}

\begin{example}
Consider the graph, 
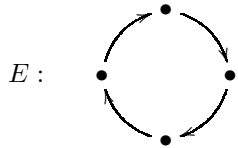
\[
\xymatrix@=13pt{
& & \bullet \ar@/^/[dr]      \\
E: &  \bullet \ar@/^/[ur] &&\bullet \ar@/^/[dl] &  \\
& & \bullet \ar@/^/[ul]     
 }
\]
One can directly show that 
\[\Mn \cong  \mathbb N\oplus \mathbb N \oplus \mathbb N \oplus \mathbb N,\] with
\[ {}^1 (a,b,c,d)=(b,c,d,a),\]
(see also~\cite[Proposition~3.7.1]{hazi}). Thus any element is periodic of order 4. Clearly for any vertex $v$, $v\in \Mn$ is a minimal element and the orbit of $v$, $O(v)=\{{}^i v,  \mid i \in \mathbb Z \}=E^0$. Note that $|O(v)|=4$ is also the length of the cycle.  In Theorem~\ref{mainthemethe} we will show that the set of cycles without exits in $E$ is in one-to-one correspondence with the orbits of minimal periodic elements of $\Mn$ and the length of cycles are equal to the order of the corresponding orbits. 
\end{example}

\begin{example}

The following example shows that the monoid $\Mn$ can have a very rich structure. Consider the graph 
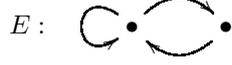
\begin{equation*}
{\def\labelstyle{\displaystyle}
 \xymatrix{
 E: &\bullet  \ar@(lu,ld)\ar@/^0.9pc/[r] & \bullet \ar@/^0.9pc/[l] 
}} 
\end{equation*}
\smallskip 
with the adjacency matrix $A_E= \left(\begin{matrix} 1 & 1\\ 1 & 0
\end{matrix}\right)$. We know that the Leavitt path algebra $L_F(E)$ is strongly graded~\cite[Theorem~1.6.15]{hazi}. Thus by Dade's theorem (\cite[Theorem~1.5.1]{hazi}) there is an equivalence of categories $\Grr L_F(E) \cong \Modd L_F(E)_0$. Here $\Grr L_F(E)$ is the category of graded modules over $L_F(E)$  and   $\Modd L_F(E)_0$ is the category of modules over the zero-component ring $L_F(E)_0$.  This implies that $K^{\gr}_0(L_F(E)) \cong K_0(L_F(E)_0)$, with the positive cones mapping to each other. The zero-component $L_F(E)_0$ is the Fibonacci algebra and its $K_0$ and its positive cone is calculated in (see~\cite[Example IV.3.6]{davidson}) which are the direct limit of 
\[ \mathbb Z\oplus \mathbb Z
\stackrel{A_E}{\longrightarrow}  \mathbb Z\oplus \mathbb Z
\stackrel{A_E}\longrightarrow  \mathbb Z\oplus \mathbb Z
\stackrel{A_E}\longrightarrow \cdots,\]
and
\[ \mathbb N\oplus \mathbb N
\stackrel{A_E}{\longrightarrow}  \mathbb N\oplus \mathbb N
\stackrel{A_E}\longrightarrow  \mathbb N\oplus \mathbb N
\stackrel{A_E}\longrightarrow \cdots,\]
respectively. 
Since $\Mn$ is the positive cone of  $K^{\gr}_0(L_F(E))$, we have 
\[\Mn = \varinjlim_{A} \mathbb N \oplus \mathbb N= \Big \{ (m,n)\in \mathbb Z \oplus \mathbb Z  \, \, \Big \vert \, \,   \frac{1+\sqrt{5}}{2} m +n \geq 0 \Big \},\]
with the action 
\[{}^1 (m,n)=(m,n) \left(\begin{matrix} 1 & 1\\ 1 & 0
\end{matrix}\right)= (m+n,m).\]

In contrast, $M_E= \{0, u\}$, where $u=u+u$.

\end{example}

Denote by $\LL^{\gr}\big(L_F(E)\big)$ the lattice of graded ideals of $L_F(E)$. There is a lattice isomorphism between the set $\TT_E$ of  hereditary saturated subsets of $E$ and the set $\LL^{\gr}\big(L_F(E)\big)$ (\cite[Theorem~2.5.9]{AAS}). The correspondence is  
 \begin{align}\label{latticeisosecideal}
 \Phi: \TT_E&\longrightarrow \LL^{\gr}\big(L_F(E)\big),\\ 
  H &\longmapsto I(H), \notag 
 \end{align} 
 where $H$ is a hereditary saturated subset and $I(H)$ is the graded ideal generated by the set $\big \{v \;|\; v\in H \big \}.$
On the other hand there is a lattice isomorphism between the set  $\TT_E$  and the lattice of $\Gamma$-order-ideals of $\Mn$~\cite[Theorem 5.11]{ahls}. The correspondence is  
 \begin{align}\label{latticeisosecideal2}
 \Phi: \TT_E&\longrightarrow \LL\big(\Mn\big),\\ 
  H &\longmapsto \langle H \rangle, \notag 
 \end{align} 
where $ \langle H \rangle$ is the order-ideal generated by the set $\big \{v \;|\; v\in H \big \}.$ Combining these two correspondence we have a lattice isomorphism
 \begin{align}\label{latticeisosecideal3}
 \Phi: \LL\big(\Mn\big) &\longrightarrow \LL^{\gr}\big(L_F(E)\big),\\ 
   \langle H \rangle &\longmapsto I(H). \notag 
 \end{align} 
Thus the Leavitt path algebra $L_F(E)$ is a graded simple ring if and only if $\Mn$ is a simple $\mathbb Z$-monoid.

We will frequently use the following two facts: The \emph{forgetful function} 
\begin{align}\label{forgthryhr}
\Mn &\longrightarrow M_E,\\
v(i) &\longmapsto v, \notag
\end{align}
relates the graded monoid to the non-graded counterpart. In several of the proofs, we pass the equalities in $\Mn$ to $M_E$ and then use Lemma~\ref{aralem6}. The other key fact is the $\mathbb Z$-module isomorphism of monoids 
\begin{align}\label{forgthryhr2}
\Mn &\longrightarrow M_{\overline E},\\
v(i) &\longmapsto v_i, \notag
\end{align}
where $\overline E$ is the covering of the graph $E$ as defined in \S\ref{graphsec}. Again the transition from the graded monoid to $M_{\overline E}$ is a crucial step in several of our proofs, as $\overline E$ is an acyclic stationary graph which repeats in each level going from level $n$ to level $n+1$ (see examples (\ref{level421}) and (\ref{level422})).

\begin{rmk}[{\bf The Abrams-Sklar treatment of the Mad Vet}]
As Alfred North Whitehead put it: ``the paradox is now fully established that the utmost
abstractions are the true weapons with which to control our
thought of concrete fact''.
In \cite{abramssklar} Abrams and Sklar demonstrated this abstract approach beautifully by realising that a recreational puzzle, called the mad veterinarian, can be answered via assigning a graph $E$ to the problem and then solving an equation (if possible) in the monoid $M_E$. We give one instance of the puzzle from \cite{abramssklar} and show how this puzzle can naturally be modified so that $\Mn$ becomes the model of the puzzle.

Suppose a Mad Veterinarian has three machines with the following properties.
\begin{itemize}
\item Machine 1 turns one ant into one beaver;

\item Machine 2 turns one beaver into one ant, one beaver and one cougar; 

\item Machine 3 turns one cougar into one ant and one beaver.

\end{itemize}

It is also supposed that each machine can operate in reverse. The puzzle now asks, for example, whether one can start with one ant and then using the machines produce 4 ants. In order to solve the puzzle, as it was observed in \cite{abramssklar}, one can naturally assign the graph $E$ below to this problem and the question then becomes whether $a=4a$ in $M_E$.
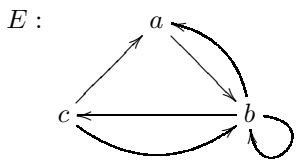
\[E:  \xymatrix{ {} & a \ar[rd] & & {} \\
c  \ar[ru]
\ar@/^{-15pt}/ [rr]&  & b
\ar@(r,d)
 \ar[ll] \ar@/^{-10pt}/
[lu] & }\]

\bigskip 
\smallskip

We modify the puzzle as follows: If the machines create new species of age one month older and we are only allowed to feed species of the same age to the machine, then the puzzle will be modelled by the graded monoid $\Mn$ instead. 
 As an example one can start with an ant, and obtain two cougars of age 1 and 2 months respectively, because in $\Mn$ we have  
\[a=b(1)=c(2)+b(2)+a(2)=c(2)+c(1)={}^{2}c+{}^{1}c.\]
\end{rmk}

\section{Orbits of the monoid $\Mn$}

  Recall the pre-ordering one can define on a monoid from \S\ref{monsec}, i.e., $b\leq a$ if $a=b+x$.  The following lemma shows that for an $n\in  \mathbb Z_{<0}$, and $a\in\Mn$, if ${}^na$ and $a$ are comparable, then either ${}^n a=a$ or ${}^n a > a$.  Here $\mathbb Z_{<0}$ is the set of negative integers and ${}^n a$ is the result of the action of $n$ on the element $a\in \Mn$.

\begin{lem}\label{nalem} 
Let $E$ be a row-finite graph. For any $a\in \Mn$, it is not possible that ${}^n a < a$, where $n\in \mathbb Z_{<0}$.  
\end{lem}

\begin{proof} Suppose  ${}^{n} a < a$. Then   
\begin{equation}\label{sunghtgrexx}
{}^na+x=a,
\end{equation} where $x\not = 0$ and $n\in \mathbb Z_{<0}$. First assume that no sinks 
appear in \emph{any} presentation of $a$.  Let $a=v^1(i_1)+\dots+v^k(i_k)$ be a presentation of $a$, where $v^s\in E^0$ and $i_s \in \mathbb Z$, $1\leq s \leq k$. 
Since by (\ref{monoidrelation2}),  $v^s(i_s)=\sum_{e\in s^{-1}(v^s)} r(e)(i_s+1)$, we can shift each of the vertices enough times so that we re-write $a$ as $a=w^1(l)+\dots+w^p(l)$, for some $l\in \mathbb Z$ and $w^s\in E^0$, $1\leq s \leq p$. 
Now without loss of generality we can assume $a=w^1+\dots+w^p$ and ${}^n a < a$. 
We now pass the equality (\ref{sunghtgrexx}) to the monoid $\Mnb$, via the isomorphism~(\ref{forgthryhr2}), where $\overline E$ is the covering graph of $E$. By the confluence property, Lemma \ref{aralem6}, there is an element $c=u^1_m+\dots + u^q_m$, where $m\geq 0$  such that
\begin{align}\label{hgfgft}
w^1_0+\dots+w^p_0 &\longrightarrow u^1_m+\dots + u^q_m,\\
w^1_n+\dots+w^p_n + x &\longrightarrow u^1_m+\dots + u^q_m.\notag
 \end{align}
 Note that $c$ is an element in the free abelian monoid  generated by  
$\overline{E}^{0}$ and since we assumed all elements in the presentation of $a$ are regular, we can arrange that all generators of $c$ appear on the ``level'' $m$.  One should visualise this  by thinking that all the elements in $a$ are sitting on level 0, ${}^n a$ on the left hand side of $a$ on level $n$ (because $n$ is negative) and $c$ on the right hand side of $a$ on the level $m$ (see Example~\ref{level421}). 

Since the graph $\overline E$ is a stationary, namely the graph repeats going from level $i$ to level $i+1$, from (\ref{hgfgft}) we get
\begin{equation}\label{hfgftgftr63}
u^1_{n+m}+\dots+u^q_{n+m} + x \longrightarrow u^1_m+\dots + u^q_m.
\end{equation}
Since in each transformation $\rightarrow_{1}$ (see (\ref{hfgtrgt655})), the number of generators either increase or stay the same (and $c$ is the sum of independent generators), 
 it is not possible  the left hand side of (\ref{hfgftgftr63}) transforms to the right hand side and therefore we can't have ${}^n a < a$.  
 
 We are left to show that indeed no sinks 
 appear in any presentation of $a$. Let  $v^1$ be a sink in the presentation  $a=v^1(i_1)+\dots+v^k(i_k)$, where, as in the argument above, 
 all the shifts appearing in ${}^n a$ are less than shifts $i_1,\dots,i_k$ in $a$. Now  ${}^n a +x= a$ implies that 
 ${}^n a +x\rightarrow c$ and $a \rightarrow c$. However, since $v^1$ is sink, there is no transformation for $v^1$ and thus $v^1(n+i_1)$ in ${}^n a$ should appear in $c$. But the shifts under the transformation $\xra_1$ either increase or stay the same. Since $a\xra c$ we have that $c$ can not recover $v_1(n+i_1)$. This is a contradiction. 
\end{proof}

The following proposition is crucial for the rest of the results in the paper. 

\begin{prop}\label{goldenprop}
Let $E$ be a row-finite graph. 

\begin{enumerate}[\upshape(i)]

\item The graph $E$ has a cycle with no exit if and only if there is an $a\in \Mn$  such that ${}^n a=a$, where $n\in \mathbb Z_{<0}$. 

\medskip 

\item The graph $E$ has a cycle with an exit if and only if there is an $a\in \Mn$  such that ${}^n a > a$, where $n\in \mathbb Z_{<0}$.

\medskip 

\item The graph $E$ is acyclic if and only if for any  $a\in \Mn$ and $n \in \mathbb Z_{<0}$, ${}^n a \parallel a$, i.e., ${}^n a$ and $a$ are not comparable.

\end{enumerate}
\end{prop}
\begin{proof}

(i) Suppose the graph $E$ has a cycle $c$ of length $n$ with no exit. Writing $c=c_1c_2\dots c_n$, where $c_i \in E^1$, since $c$ has no exit,  the relations in $\Mn$ show that $v=s(c_1)=r(c_1)(1)=r(c_2)(2)=\cdots =r(c_n)(n)=v(n)$. It follows that ${}^n v=v$ in  $\Mn$ (i.e. ${}^{-n} v=v$). Alternatively, one can see that ${}^{-1} a=a $, where $a=\sum_{i=1}^n r(c_i) \in \Mn$. 

Conversely, suppose  that there is an $a\in \Mn$ and $n \in \mathbb Z_{<0}$ such that  ${}^n a = a$. 
Let $a=v^1(i_1)+\dots+v^k(i_k)$  be a presentation of $a$, where $v^s\in E^0$ and $i_s \in \mathbb Z$. Similar to the proof of Lemma~\ref{nalem}, we first assume that no sinks appear in \emph{any} presentation of $a$. 
Thus $a=v^1(i_1)+\dots+v^k(i_k)$, where all $v^s$'s are regular. Since $v^s(i_s)=\sum_{e\in s^{-1}(v^s)} r(e)(i_s+1)$, we can shift each of the vertices enough times so that we re-write $a$ as $a=w^1(l)+\dots+w^p(l)$, for some $l\in \mathbb Z$ and $w^s\in E^0$, $1\leq s \leq p$. 
Now without loss of generality we can assume $a=w^1+\dots+w^p$ and ${}^n a=a$. 

We now pass $a$ to the monoid $\Mnb$, via the isomorphism~(\ref{forgthryhr2}), where $\overline E$ is the covering graph of $E$ which is acyclic by construction. Thus we have $w^1_0+\dots +w^p_0= w^1_n+\dots +w^p_n$ in $\Mnb$, where all $w^s_k \in \overline E^0$. By Lemma~\ref{aralem6}, there is an element $c=u^1_m+\dots + u^q_m$, where $m\geq 0$  such that $w^1_0+\dots +w^p_0 \rightarrow c$ and 
$w^1_n+\dots +w^p_n \rightarrow c$. Note that since we assumed all elements in the presentation of $a$ are regular, we can arrange that all generators of $c$ appear on the ``level'' $m$. Since the graph $\overline E$ is a stationary, namely the graph repeats going from level $i$ to level $i+1$, we have 
\[w^1_n+\dots +w^p_n \longrightarrow u^1_{m+n}+\dots u^q_{m+n},\] and consequently 
 \begin{equation}\label{subgdtgee1}
 u^1_{m+n}+\dots + u^q_{m+n} \longrightarrow u^1_m+\dots + u^q_m.
 \end{equation}
 Since the number of generators on the right and the left hand side of (\ref{subgdtgee1}) are the same and the relation $\rightarrow$ would increase the number of generators if there is  more than one edge emitting from a vertex, it follows that there is only one edge emitting from each vertex in the list $A=\{u^1_{m-n}, \dots , u^q_{m-n}\}$ and their subsequent vertices until the edges reach the list $B=\{u^1_m, \dots , u^q_m\}$. Thus we have a bijection $\rho: A\rightarrow B$. Consequently,  there is an $l\in \mathbb N$ such that $\rho^l=1$. Then we have $u^i_{m+n} \rightarrow u^i_{m-ln}$ for all elements of $A$. Going back to the graph $E$, this means there is a path with no bifurcation from $u$ to itself, namely there is a cycle with no exit based at $u$. 
 
 We are left to show that indeed no sinks appear in any presentation of $a$. Let $a=v^1(i_1)+\dots+v^k(i_k)$ be a presentation of $a$ with $v^1$ a sink. Since ${}^n a = a$, we can choose $n<0$ small enough that all the shifts appearing in ${}^n a$ are smaller than shifts $i_1,\dots,i_k$ in $a$. Now  ${}^n a = a$ implies that 
 ${}^n a \rightarrow c$ and $a \rightarrow c$. 
 However, since $v^1$ is sink, there is no transformation for $v^1$ and thus $v^1(n+i_1)$ should appear in $c$. But the shifts appearing in $c$ are bigger than those in $a$ as $a$ also transforms to $c$ which can't happen. Thus there are no sinks in a presentation of $a$.  

 \medskip 
 
(ii) Suppose the graph $E$ has a cycle $c=c_1c_2\dots c_n$ of length $n$ with exits.  Consider $a=\sum_{i=1}^n r(c_i) \in \Mn$. Now applying the transformation rule on each $r(c_i)$ we have $a= \sum_{i=1}^n r(c_i)(1)+x=a(1)+x$, where $x\not =0$ as the cycle has an exit and thus it branches out and other symbols appear in the transformation.   This shows ${}^{-1} a>a$ as claimed. 

Conversely, suppose  that there is an $a\in \Mn$ and $n \in \mathbb Z_{<0}$ such that  ${}^n a > a$.    
Let $a=v^1(i_1)+\dots+v^k(i_k)$ be a presentation of $a$, where $v^s\in E^0$ and $i_s \in \mathbb Z$. We can  choose a positive number $n$ such that ${}^n a < a$.  Further, we can choose $n$ big enough such that all the shifts appearing in ${}^n a$ are bigger than the shifts $i_1,\dots, i_k$ in $a$. We then have $a={}^n a +x $, for some nonzero 
$x\in \Mn$. We pass the equality to the monoid of the covering graph $\overline E$, via the isomorphism~(\ref{forgthryhr2}). By the confluence property, Lemma~\ref{aralem6}, there is $c$ such that $a\rightarrow c$ and ${}^n a +x \rightarrow c$. Using Lemma~\ref{aralem6}, we can then write $c=d+f$, where ${}^n a \rightarrow d$ and 
$x \rightarrow f$. Suppose $d=u^1(j_1)+\dots+u^t(j_t)$. Since the graph $\overline E$ is stationary, applying the same transformations done on ${}^n a$ to $a$ we obtain $a\rightarrow d'$, where $d'=u^1(j'_1)+\dots+u^t(j'_t)$ and $d' \rightarrow c$. 
Putting these together we have 
\begin{equation}\label{hghgyhuhrfe}
u^1(j'_1)+\dots+u^t(j'_t) \longrightarrow u^1(j_1)+\dots+u^t(j_t) +f. 
\end{equation}
Let $A_0$ be the list $\{u^1(j'_1),\dots,u^t(j'_t)\}$ appearing on the left hand side of (\ref{hghgyhuhrfe}) and $B_0$ the list $\{u^1(j_1),\dots,u^t(j_t)\}$ on the right hand side.  All the paths emitting from $A_0$ ends up in either $B_0$ or the list of vertices in $f$. Since the number of vertices (symbols) on the right hand side of (\ref{hghgyhuhrfe}) is more than the left hand side, some of the paths emitting from the left hand side has to have bifurcation.  Now let $A_1:=\phi^{-1}(B_0)$ denote all the vertices in $A_0$ which has a path ending in $B_0$. Clearly $A_1 \subseteq A_0$. Similarly there are paths coming from $A_1$ which have bifurcation. Denote by $B_1\subseteq B_0$ the list of same vertices (with possibly different shifts) which appear in $A_1$. Consider $A_2:=\phi^{-1}(B_1)$ which is $A_2\subseteq A_1$, i.e., vertices in $A_1$ which has a path ending in $B_1$. Repeating this process, since $A_0$ is finite, there is a $k$ such that $A_{k+1}=\phi^{-1}(B_k)$ and $A_{k+1}=A_{k}$. Note that $A_k$ and $B_k$ consist of same vertices (with possibly different shifts) and for each vertex in $A_k$ there is a path ending in $B_k$ and for each vertex in $B_k$ there is a \emph{unique} path from a vertex in $A_k$ to this vertex. Since $A_k$ is finite, this assignment defines a bijective function from $\rho: A_k \rightarrow B_k$. Thus there is an $l\in \mathbb N$ such that $\rho^l=1$. Since the graph $\overline E$ is stationary, this means after $l$ repeat, all paths return to the same vertices they started from, i.e., there are cycles in the graph. However, as some of the paths had to have bifurcation, there exists cycles with exits.

\medskip 

(iii)  Suppose $E$ is acyclic. For any $a\in \Mn$ and $n\in \mathbb Z_{<0}$,  we can't have ${}^na > a$ or ${}^na = a$,  otherwise by (i) and (ii) of the theorem,  $E$ has cycles.  On the other hand by Lemma \ref{nalem}, it is also not possible to have  ${}^na < a$ for any $n\in \mathbb Z_{<0}$.  Thus  ${}^na$ and $a$ are not comparable.

Conversely, suppose no elements of  $\Mn$ are comparable within its orbits. Then (i) and (ii) immediately imply that $E$ is acyclic. 
\end{proof}

Recall that if a group $\Gamma$ acts on a set $M$, then the action is \emph{free} if ${}^\gamma m=m$ implies that $\gamma$ is the identity of the group, i.e., all the isotropy groups of the action are trivial. 

\begin{cor}\label{conLm}
Let $E$ be a row-finite graph. 

\begin{enumerate}[\upshape(i)]

\item The graph $E$ satisfies Condition (L) if and only if $\mathbb Z$ acts freely on $\Mn$.


\medskip 

\item The graph $E$ satisfies Condition (K) if and only if $\mathbb Z$ acts freely on any quotient of $\Mn$ by an order-ideal.

\end{enumerate}
\end{cor}

\begin{proof}

(i)  Suppose the graph $E$ satisfies Condition (L). If ${}^n a= a$, for some $n \in \mathbb Z_{<0}$, and $a\in \Mn$, then $E$ has a cycle without exit by Proposition~\ref{goldenprop} which is not possible. Thus $\mathbb Z$ acts freely on $\Mn$. 

Conversely, if $E$ has a cycle without exit, then there is an $a\in \Mn$ and $n \in \mathbb Z_{<0}$ such that  ${}^n a=a$ (Proposition~\ref{goldenprop}) which contradicts the freeness of the action. 

\medskip 

(ii)  Recall that the graph $E$ satisfies Condition (K) if and only if for every hereditary saturated subset $H$ the quotient graph $E/ H$ satisfies Condition (L). 

Suppose $E$ satisfies condition (K) and $I$ is an order-ideal of $\Mn$. Then there is a hereditary saturated subset $H\subseteq E^0$ which generates $I$. Since by Lemma~\ref{qmiso},  there is a $\mathbb Z$-module isomorphism  $\Mn /I \cong M^{\gr}_{E / H}$, and $E / H$ has condition (L), by part (i), $\mathbb Z$ acts freely on $M^{\gr}_{E / H}$ and thus on 
$\Mn /I$. 

Conversely, for any hereditary saturated subset $H$, and the corresponding order-ideal $I$ of $\Mn$, since $\mathbb Z$ acts freely on $\Mn /I \cong M^{\gr}_{E / H}$, it follows by part (i) that $E/ H$ has condition (L). It then follows that $E$ has condition (K). 
\end{proof}

\section{The ideal structure of $L_F(E)$ via the monoid  $\Mn$} \label{seciv}

There is an alluring proposition  in the book of Abrams, Ara and Molina Siles~\cite[Proposition~6.1.12]{AAS}, which states that for a finite graph $E$, the Leavitt path algebra $L_F(E)$ is purely infinite simple if and only if $M_E \backslash \{0\}$ is a group. As the monoid $\Mn$ is a much richer object than $M_E$, it is expected to capture more of the structure of the Leavitt path algebra $L_F(E)$.    

We start with an immediate corollary of the results on the periodicity of elements of $\Mn$ established in the previous section.

\begin{cor}\label{conKm}
Let $E$ be a row-finite graph and $L_F(E)$ its associated Leavitt path algebra. 

\begin{enumerate}[\upshape(i)]

\item The algebra $L_F(E)$ is graded simple if and only if $\Mn$ is simple. 

\medskip 

\item The algebra $L_F(E)$ is simple if and only if $\Mn$ is simple and for any $a\in \Mn$, if ${}^n a$ and $a$ are comparable, $n \in \mathbb Z_{<0}$, then ${}^n a > a$. 

\medskip 

\item The algebra $L_F(E)$ is purely infinite simple if and only if $\Mn$ is simple, if ${}^n a$ and $a$ are comparable, $n \in \mathbb Z_{<0}$, then ${}^n a > a$ and there is 
an $a\in \Mn$ such that ${}^n a > a$. 

\medskip 

\item 
If $E$ is a finite graph, then $L_F(E)$ is purely infinite simple if and only if $\Mn$ is simple and for any $a\in \Mn$ there is $n \in \mathbb Z_{<0}$ such that  ${}^n a > a$.

\end{enumerate}
\end{cor}
\begin{proof}

(i) This follows from the fact that there is a lattice isomorphism between the graded ideals of $L_F(E)$ and $\mathbb Z$-order-ideals of  $\Mn$  via the lattice of hereditary saturated subsets of  the graph $E$ (see \eqref{latticeisosecideal3}).  

\medskip 

(ii) The algebra $L_F(E)$ is simple if and only if $E$ has no non-trivial hereditary saturated subsets and $E$ satisfies Condition (L) (\cite[Theorem~2.9.1]{AAS}). By Corollary~\ref{conLm}, $E$ has Condition (L) if and only if  for any $a\in \Mn$ and $n\in \mathbb Z$, ${}^n a\neq a$. It follows that if 
${}^n a$ and $a$ are comparable, $n \in \mathbb Z_{<0}$, then either ${}^n a > a$ or ${}^n a <a$. The proof now follows from combining this fact with part (i) and Lemma \ref{nalem}.

\medskip

(iii) The algebra $L_F(E)$ is purely infinite simple if and only if $E$ has no non-trivial hereditary saturated subsets, $E$ satisfies Condition (L) and $E$ has at least one cycle with an exit (\cite[Theorem~3.1.10]{AAS}). The proofs now follow from combining this fact with part (i) and (ii) and Corollary~\ref{conLm}. 

\medskip

(iv) Suppose that $L_F(E)$ is purely infinite simple. We need to prove that for any $a\in\Mn$, ${}^n a$ and $a$ are comparable for some $n\in \mathbb Z_{<0}$. 
We claim that for any $v\in E^0$ there is an $n\in\mathbb Z_{<0}$ such that ${}^n v>v$.  For any $v\in E^0$, let $X_v=\{p\;|\; s(p)=v,  p \text{~is a finite path not containing a cycle}, \text{~and~} r(p) \text{~is on a cycle}\}$. Then for each $p\in X_v$, the length of $p$ is at most $|E^0|$ as any path of length large than $|E^0|$ contains a cycle. Then the maximal number in $\{l(p)\;|\; p\in X_v\}$ is less than or equal to $|E^0|$. We have a representation $v=\sum_{i=1}^n v_i(s_i)$ such that all $v_i$'s are on cycles $C_1, \cdots, C_n$ (possibly $C_i=C_{i'}$ with $i\neq i'$) by induction on the maximal length of paths in $X_v$. 
Let $l_i$ denote the length of the cycle $C_i$. Then $v_i >v_i(l_i)$ for $1\leq i\leq n$ as each cycle among $C_1,\cdots, C_n$ has exits.
Let $l$ be the minimal common multiple of $l_1, \cdots, l_n$. Then $v_i >v_i(l)$ for $1\leq i\leq n$ and thus $v>v(l)$, i.e. ${}^{-l}v>v$ as claimed. Now for any $a\in \Mn$, $a=\sum_{s=1}^ku_t(j_t)$ (possibly $v_t=v_{t'}$ if $t\neq t'$). As for each vertex $u_t$ there exists $m_t$ such that ${}^{-m_t}u_t> u_t$. Take $n=-\prod_{t=1}^k m_t$. Then we have ${}^{n} a={}^{-\prod_{t=1}^k m_t}\sum_{t=1}^k u_t(j_t)>\sum_{t=1}^k u_t(j_t)=a$. Conversely, the proof follows from (iii).

\end{proof}


These results show that acyclicity/simplicity or purely infinite simplicity of algebras are preserved under an order-isomorphism between their graded Grothendieck groups. 

Although the monoid $\Mn$ is constructed from the graded projective modules, it can however detect the non-graded structure of a Leavitt path algebra. Here is the first instance in this direction.  We will produce more evidence of this in Theorem~\ref{mainthemethe}. 

\begin{prop}\label{propnongra}
Let $E$ be a row-finite graph and $L_F(E)$ its associated Leavitt path algebra. Then $L_F(E)$ has a non-graded ideal if and only if there is an order-ideal $I$ of $\Mn$ such that the quotient monoid $\Mn/I$ has a periodic element. 
\end{prop}
\begin{proof}

Suppose $L_F(E)$ has a non-graded ideal. Then there is a hereditary saturated subset $H$ and a non-empty set $C$ of cycles which have all their exits in $H$ (see \cite[Proposition 2.8.11]{AAS}). Thus $E/ H$ is a graph for which cycles in $C$ have no exit. By Proposition~\ref{goldenprop}, $M^{\gr}_{E / H}$ has periodic elements. But by Lemma~\ref{qmiso}, 
$M^{\gr}_{E / H}  \cong \Mn / I$, where $I$ is the order-ideal generated by $H$ and so it has periodic elements.  The converse argument is similar. 
\end{proof}

For $a\in \Mn$, we denote the smallest $\mathbb  Z$-order-ideal generated by $a$ by $\langle a \rangle $. It is easy to see that 
\begin{equation}\label{oridealhj}
\langle a \rangle=\big \{ x \in \Mn \mid x \leq \sum_{} {}^i a \big \}. 
\end{equation}
Thus for a vertex $v$, denoting $\overline v$ for the smallest hereditary saturated subset containing $v$, we have 
\begin{equation}\label{oridealhj2}
\langle v \rangle=\langle \overline v \rangle.   
\end{equation}

The following lemmas show how these ideals of the monoid capture the geometry of the graph.  Recall the notion of ``local'' cofinality from \S\ref{graphsec}. The proof of the following lemma is a ``local'' version of \cite[Lemma 2.9.6]{AAS} and we leave it to the reader. 

\begin{lem}\label{localcofin}
Let $E$ be a row-finite graph. For $v,w\in E^0$ the following are equivalent. 

\begin{enumerate}[\upshape(i)]

\item $\langle w \rangle  \subseteq \langle v \rangle$;

\medskip

\item the vertex $v$ is cofinal with respect to $w\in E^0$;
\medskip

\item if $v \in H$, then $w\in H$, where $H$ is a hereditary saturated subset of $E$. 

\end{enumerate}
\end{lem}



\begin{lem}\label{dowdirtwo}
Let $E$ be a row-finite graph. 

\begin{enumerate}[\upshape(i)]

\item For $u,v\in E^0$, $\langle u \rangle \cap \langle v \rangle \not = 0$ if and only if $u$ and $v$ are downward directed. 

\medskip

\item The vertex $v$ is cofinal if and only if $\langle v \rangle=\Mn$.

\medskip 

\item $E$ is cofinal if and only if $\langle v \rangle=\Mn$ for every $v\in E^0$. 

\end{enumerate}
\end{lem}

\begin{proof}

(i) Suppose $u$ and $v$ are downward directed, i.e., there is an $w\in E^0$ such that $v\geq w$ and $u\geq w$. Thus there is a path $\alpha$ with $s(\alpha)=v$ and $r(\alpha)=w$. This gives that $v=w(k)+t$ for some $k\in \mathbb Z$ and $t\in \Mn$. Thus $w \in \langle v \rangle$. Similarly $w\in \langle u \rangle$ and so  $\langle u \rangle  \cap \langle v \rangle \not = 0$. 

Conversely, suppose $0 \not= a\in \langle u \rangle   \cap \langle v \rangle$. Since $a$ is a sum of vertices (with given shifts), and $\langle v \rangle$ and $\langle u \rangle$ are order-ideals, one can find a vertex $z \in  \langle u \rangle\cap \langle v \rangle$. 
Thus $\sum_{i} {}^i v = z+t$ and $\sum_{j} {}^j u = z+s$ in $\Mn$. Passing to $M_E$ via the forgetful function~\ref{forgthryhr}, we have $nv=z+ t'$ and $mu=z+s'$, for  $t',s' \in M_E$ and $m,n\in \mathbb N$. By the confluence property, Lemma~\ref{aralem6}, there are $c, d \in F_E$ such that 
$nv \rightarrow c$, $z+t' \rightarrow c$ and $mu \rightarrow d$, $z+s' \rightarrow d$. By Lemma~\ref{aralem6}, one can write $c=c_1+c_2$ and $d=d_1+d_2$, such that $z\rightarrow c_1$, $t'\rightarrow c_2$ and $z\rightarrow d_1$, $t'\rightarrow d_2$. Since in $M_E$, we have $z=c_1=d_1$, again by the confluence property, there is an $e\in F_E$ such that $c_1\rightarrow e$ and $d_1\rightarrow e$. Hence $z\rightarrow e$. Now 
\begin{align*}
nv &\longrightarrow c_1+c_2 \longrightarrow e+c_2\\
mu &\longrightarrow d_1+d_2 \longrightarrow e+d_2.
\end{align*}
One more use of \ref{aralem6} shows that all the vertices appearing in $e$ are in both the tree of $v$ and the tree of $u$. Thus $u$ and $v$ are downward directed. 

\medskip

(ii) This follows immediately from Lemma~\ref{localcofin}.

\medskip

(iii) This follows from part (ii). 
\end{proof}

For a $\Gamma$-monoid $M$, a $\Gamma$-order-ideal $N \subseteq M$ is called \emph{prime} if for any $\Gamma$-order-ideals $N_1, N_2 \subseteq M$, $N_1\cap N_2 \subseteq N$ implies that $N_1\subseteq N$ or $N_2\subseteq N$.  In case of $\Mn$,  the lattice isomorphism~(\ref{latticeisosecideal3}), immediately implies that prime order-ideals of $\Mn$ are in one-to-one correspondence with the graded prime ideals of $L_F(E)$. 

Recall that one can give an element-wise description for a prime ideal of a ring. Namely, an ideal $I$ of a ring $A$ is prime if for any $a, b\not \in I$, there is an $r\in R$ such that $arb \not \in I$. We have a similar description in the setting of $\Mn$ demonstrating how the monoid structure of $\Mn$ mimics the algebraic structure of $L_F(E)$.

\begin{lem} \label{restrictedlatticeiso} Let $E$ be a row-finite graph. Then the order-ideal $I$ of $\Mn$ is prime if and only if for any $a,b \not \in I$, there is a $c \not \in I$ and $n,m \in \mathbb Z$, such that ${}^n c \leq a$ and ${}^m c \leq b$.
\end{lem}

\begin{proof}

$\Rightarrow$ Suppose $I$ is a prime order-ideal. A combination of correspondence~(\ref{latticeisosecideal3}) and  \cite[Proposition~4.1.4]{AAS} give that $I= \langle H \rangle$, where $H$ is a hereditary saturated subset such that $E^0 \backslash H$ is downward directed. Let $a,b \in \Mn$ such that $a \not \in I$ and $b\not \in I$. Since $a=\sum v_i(k_i)$ and $b=\sum w_j(k'_j)$ are sum of vertices (with given shifts), and $I$ is an order-ideal, then  (possibly after a re-arrangement) $v_1 \not \in H$ and $w_1 \not \in H$. Thus there is a $z\not \in H$ such that $v_1\geq z$ and $w_1\geq z$. Thus there is a path $\alpha$ with $s(\alpha)=v_1$ and $r(\alpha)=z$. This shows that in $\Mn$, for some $i\in \mathbb Z$, ${}^i z \leq v$ and consequently ${}^{i+k_1}z  \leq v_1(k_1) \leq a $. Similarly for a $j \in \mathbb Z$, ${}^j z \leq w_1$ and consequently ${}^{j+k'_1}z  \leq w_1(k'_1) \leq b$.

$\Leftarrow$ Suppose $I_1$ and $I_2$ are order-ideals such that $I_1\cap I_2 \subseteq I$. If $I_1\nsubseteq I$ and $I_2\nsubseteq I$, then there are $a \in I_1 \backslash I$ and $b \in I_2 \backslash I$. By the property of $I$, there is a $c\not \in I$ such that  ${}^n c \leq a$ and ${}^m c \leq b$. Since $I_1$ and $I_2$ are order-ideals $c\in I_1\cap I_2$ and thus $c\in I$ a contradiction. Thus $I$ is prime.  
\end{proof}

Recall the notions of a line-point from \S\ref{graphsec} and the minimal elements of monoids from \S\ref{monsec}. By~\cite[Proposition~2.6.11]{AAS}, a minimal left ideal of a Leavitt path algebra $L_F(E)$ is isomorphic to $L_F(E)v$, where $v$ is a line-point. Here we show that we can distinguish these vertices in the monoid $\Mn$.

\begin{lem}\label{tminimi}
Let $E$ be a row-finite graph and $L_F(E)$ its associated Leavitt path algebra. 
\begin{enumerate} [\upshape(i)]

\item The vertex $v$ has no bifurcation if and only if $v\in \Mn$ is minimal.

\medskip 

\item The vertex $v$ is a line-point if and only if  $v\in \Mn$ is minimal and aperiodic. 

\medskip 

\item The left ideal $L_F(E)v$ is minimal if and only if $v\in \Mn$ is minimal and aperiodic. 

\end{enumerate}
\end{lem}
\begin{proof}

(i) Suppose $v\in E^0$ has no bifurcation. If $v\in \Mn$ is not minimal then there is an $a \in \Mn$ such that $a < v$. Thus $a+x=v$ for some $x\not =0$. 
Passing the equality to $M_E$ by the forgetful function~\ref{forgthryhr},  and invoking the confluence property of $M_E$, Lemma~\ref{aralem6}, we get an $c \in F$ such that $a+x \rightarrow c$ and $v\rightarrow c$. Since $v$ has no bifurcation, $c$ has to be a vertex and thus $a$ has to be this vertex and $x=0$ which is a contradiction. 

Conversely, suppose that  $v\in \Mn$ is minimal.  If $v$ has a bifurcation, then pick the first $u \in T(v)$,  where this bifurcation occurs. Then 
\[v=u(k)=\sum_{\alpha \in s^{-1}(u)} r(\alpha)(k+1),\] where $k\in \mathbb N$ is the length of the path connecting $v$ to $u$.  Since  $|s^{-1}(u)| >1$ then $r(\alpha)(k+1) < v$ which is a contradiction. 

\medskip 

(ii) Suppose $v$ is a line-point. Then by (i), $v \in \Mn$ is minimal. If $v$ is periodic, then there is $n<0$ such that ${}^n v =v$. Passing the equality to $\overline E$ via the isomorphism~(\ref{forgthryhr2}),  since $v_0, v_n \in \overline E^0$
 are also line-points, using the congruence~(\ref{monoidrelation}) for this case,  we have $v_n\rightarrow w_{l+n}$ and $v_0\rightarrow w_{l+n}$ for n $w\in E^0$ and some $l \geq 0$. Since $\overline E$ is stationary, we have $v_0 \rightarrow w_l$ and $w_{l+n} \rightarrow w_{l}$. This implies that $w$ is on a cycle and thus $v$ connects to a cycle which is a contradiction. Thus $v$ is aperiodic. 
 
Conversely, suppose that $v\in \Mn$ is minimal and aperiodic. By part (i) $v$ has no bifurcation. If $v$ connects to a cycle, then we have 
\[v=w(k)=w(k+l),\] where $k$ is the length of the path connecting $v$ to $w$, the based of the cycle, and $l$ is the length of the cycle. It follows that ${}^{-l} v=v$ which is a contradiction. Thus $v$ is a line-point. 

\medskip

(iii) By \cite[Proposition~2.6.11]{AAS}, the left ideal $L_F(E)v$ is minimal if and only if $v$ is a line-point. Part (ii) now completes the proof. 
\end{proof}

%

We close the paper with the following theorem which justifies why a conjecture such as Conjecture~\ref{conj1} could be valid. 

\begin{thm}\label{mainthemethe}
Let $E_1$ and $E_2$ be row-finite graphs. Suppose there is a $\mathbb Z$-module isomorphism $\phi: M^{\gr}_{E_1} \rightarrow M^{\gr}_{E_2}$. Then we have the following: 

\begin{enumerate} [\upshape(i)]

\item $E_1$ has Condition (L) if and only if $E_2$ has condition (L).

\medskip

\item $E_1$ has Condition (K) if and only if $E_2$ has condition (K).

\medskip

\item there is a one-to-one correspondence between the graded  ideals of $L_F(E_1)$ and $L_F(E_2)$. 

\medskip 

\item there is a one-to-one correspondence between the non-graded  ideals of $L_F(E_1)$ and $L_F(E_2)$. 

\medskip

\item there is a one-to-one correspondence between the isomorphism classes of minimal left/right ideals of $L_F(E_1)$ and $L_F(E_2)$. 

\end{enumerate}
\end{thm}

\begin{proof}

(i) and (ii) follows from Corollary~\ref{conLm}. 

\medskip

(iii) The correspondence between the graded ideals of Leavitt path algebras follows from~(\ref{latticeisosecideal3}) and that a $\mathbb Z$-module isomorphism of modules induces a lattice isomorphism between the set of their order-ideals.  

\medskip

(iv) We first show that the set of cycles without exits in $E$ is in one-to-one correspondence with the orbits of minimal periodic elements of $\Mn$. Notice that a minimal element $a$ of $\Mn$ can be represented by some $v(k)$, for an $v\in E^0$ and $k\in \mathbb Z$. Consider the set $S$ of all minimal and periodic elements of $\Mn$ and partition this set by their orbits, i.e., $S=\sqcup_{v\in S} O(v)$, where $O(v)=\{{}^i v,  \mid i \in \mathbb Z \}$. Let $C$ be the set of all cycles without exits in $E$. For a cycle $c\in C$, denote by $c_v$ a vertex on the cycle (there is no need to fix this base vertex). We show that there is a bijection between the sets 
\begin{align*}
\phi: C &\longrightarrow \{O(v) \mid v \in  S \},\\
c &\longmapsto O(c_v)
\end{align*}
such that the length of the cycle $c$, $|c|$, is the same as the order of the corresponding orbit, $|O(c_v)|$. 
First note that for a cycle $c$ without exit, $c_v$ is a vertex with no bifurcation, thus by Lemma~\ref{tminimi}, $c_v$ is minimal in $\Mn$. Furthermore, one can observe that 
${}^{|c|} c_v =c_v$ in $\Mn$, which also shows that $|O(c_v)| = |c|$. Further, choosing another base point $w$ on the cycle $c$, we have $O(c_v)=O(c_w)$ and thus the map $\phi$ is well-defined. 
Now suppose $c$ and $d$ are two distinct cycles without exits. If $O(c_v)=O(d_v)$, then there is a $k\in \mathbb Z$ such that for the vertices $c_v$ and $d_v$  we have ${}^k c_v=d_v$ in $\Mn$.  Passing to $M_E$ via the forgetful function~\ref{forgthryhr}, the equation $c_v=d_v$ implies that $T(c_v) \cap T(d_v) \not = \emptyset$, which can't be the case. Thus the map $\phi$ is injective. On the other hand, if $v$ is minimal and periodic, then,  by Lemma~\ref{tminimi}, and its proof of part (ii), $v$ has no bifurcation and is connected to a cycle $c$ without an exit. We show that $O(c_v)=O(v)$. Since $v$ has no bifurcation, $v$ connects to $c_v$ by a path $\alpha$ which has no exit. Thus ${}^{-|\alpha|} v= c_v$ and the claim follows. This shows that $\phi$ is also surjective. 

Applying this argument for the quotient graph $E/ H$, for a hereditary and saturated subset $H$, gives a one-to-one correspondence between cycles whose exits are in $H$ and the minimal periodic elements of $\Mn / I$, where $I =\langle H \rangle $.

Now since a $\mathbb Z$-module isomorphism of modules preserves the class of orbits of minimal and periodic elements, an isomorphism of graded monoids of $E_1$ and $E_2$ gives a one-to-one correspondence between the cycles without exits (of the same length) between the graphs $E_1$ and $E_2$. Furthermore combining this with part (iii), we obtain a one to one correspondence between cycles whose exits are in $H$ in $E_1^0$ and the cycles whose exits are in the corresponding hereditary saturated subset in $E^0_2$. 

Finally, by the Structure Theorem for ideals (\cite[Proposition 2.8.11]{AAS}, \cite{rangaswamy2}), the non-graded ideals in a Leavitt path algebra are characterised by the ``internal data'' of hereditary saturated subsets $H$, a non-empty set $C$ of cycles whose exits are in $H$ and the ``external'' data of polynomials $p_c(x) \in F[x], c\in C$. Combining this with the correspondences established above, completes the proof. 

\medskip

(v)  Consider the set $S$ of all minimal and aperiodic elements of $\Mn$ and partition this set by their orbits, i.e., $S=\sqcup_{v\in S} O(v)$. We leave it to the reader to observe that there is a one-to-one correspondence between these orbits of the graph $E$ and the class of minimal left/right ideals $L_F(E)$ by invoking Lemma~\ref{tminimi} and \cite[Proposition~2.6.11]{AAS} that a minimal left ideal of a Leavitt path algebra $L_F(E)$ is isomorphic to $L_F(E)v$, where $v$ is a line-point. The proof now follows from the fact that a $\mathbb Z$-module isomorphism of modules preserves the class of orbits of minimal and aperiodic elements. 
\end{proof}

Note that Theorem~\ref{mainthemethe} is in line with what we should obtain from the (conjectural) statement that if $M^{\gr}_{E_1}\cong M^{\gr}_{E_2}$  as a $\mathbb Z$-modules (or equivalently $K_0^{\gr}(L(E_1) \cong K_0^{\gr}(L(E_2)$) then the Leavitt path algebras $L_F(E_1)$ and $L_F(E_2)$ are graded Morita equivalent. Indeed by \cite[Theorem~2.3.8]{hazi}, the graded Morita equivalence lifts to Morita equivalence:  
\begin{equation}
\xymatrix{
\Grr L_F(E_1) \ar[rr]  \ar[d]_{U}&& \Grr  L_F(E_2) \ar[d]^{U}\\
\Modd  L_F(E_1) \ar[rr]  && \Modd  L_F(E_2).
}
\end{equation}
Combining this with the consequences of Morita theory, we have a one-to-one correspondence between the ideals of $L_F(E_1)$ and $L_F(E_2)$ as confirmed by Theorem~\ref{mainthemethe}.

\begin{rmk} 
\label{rmk}
A graded version of $M_E$ \cite [\S 5C]{ahls} was defined to be the
abelian monoid generated by $\{v(i) \mid v\in E^0, i\in \mathbb Z\}$ subject to the relations
\begin{equation}\label{monoidrelation2}
v(i)=\sum_{e\in s^{-1}(v)}r(e)(i-1), 
\end{equation}
for every $v\in E^{0}$ that is not a sink. We denote it by ${\Mn}'$. There are $\mathbb Z$-module isomorphisms \begin{align}\label{yhoperagen}
{\Mn}' &\cong  M_{\overline{E}} \cong \mathcal V(L_F(\overline E))\cong\, \mathcal V^{\gr}(L_F(E)),
\\v(i) &\longmapsto v_i\longmapsto  L_F(\overline E) v_i    \longmapsto\big (L_F(E)v\big) (-i), \notag
\end{align} see \cite[Proposition~5.7]{ahls}. We correct here that the isomorphism ${\Mn}' \cong M_{\overline{E}}$ should be given by $v(i)\mapsto v_i$ and that  the $\mathbb Z$-action on ${\Mn}'$ given by Equation (5-10) in \cite{ahls} should be ${}^n {v(i)}=v(i-n)$ for $n, i\in\mathbb Z$ and $v\in E^0$.

In this note the relations for the graded monoid $\Mn$ are slightly different from that for ${\Mn}'$. However $\Mn$ and ${\Mn}'$ are isomorphic via $v(i)\mapsto v(-i)$ as $\Z$-modules. Note that $\Mn$ has a natural $\Z$-action ${}^n v(i)=v(i+n)$.  If we follow the notation of covering graph in \cite[\S 5B]{ahls}, we can obtain $\Mn\cong \VV^{\gr}(L(E))$ similarly as \cite[Propositon 5.7]{ahls}.  
\end{rmk}

\section{Acknowledgements} The authors would like to acknowledge Australian Research Council grant DP160101481. They would like to thank Anthony Warwick from Western Sydney University who enthusiastically took part in the discussions related to this work.


\begin{thebibliography}{9999}


\bibitem{abrams2005} G. Abrams, G. Aranda Pino, {\it The Leavitt path algebra of a graph}, Journal of Algebra $\mathbf{293}$ (2005), 319--334.

\bibitem{AAS} G. Abrams, P. Ara, M. Siles Molina, Leavitt path algebras. Lecture Notes in Mathematics, vol. 2191, Springer Verlag, 2017.

 \bibitem{abramssklar}  G. Abrams, J. Sklar,  {\it The graph menagerie:  Abstract algebra and the Mad Veterinarian},  Mathematics Magazine $\mathbf{83(3)}$, 2010, 168--179. 




\bibitem{ara2006} P. Ara, M.A. Moreno, E. Pardo, {\it Nonstable K-theory for graph algebras}, Algebras and Representation Theory
$\mathbf{10}$ (2007), 157--178.

\bibitem{arapardo} P. Ara, E. Pardo, {\it Towards a K-theoretic characterization of graded isomorphisms between Leavitt path algebras}, J. K-Theory,  {\bf 14} (2014), no. 2, 203--245.




\bibitem{ahls} P. Ara, R. Hazrat, H. Li, A. Sims, {\it Graded Steinberg algebras and their representations},  Algebra \& Number theory, {\bf 12-1} (2018), 131--172.


\bibitem{davidson} K. Davidson, $C^*$-algebras by example. Fields Institute Monographs, 6. American Mathematical Society, Providence, RI, 1996.



\bibitem{roozbehhazrat2013} R. Hazrat, {\it The graded Grothendieck group and the classification of Leavitt path algebras}, Math. Annalen $\mathbf{355}$ (2013), 273--325.




\bibitem{haz3} R. Hazrat, {\it The dynamics of Leavitt path algebras}, J. Algebra $\mathbf{384}$ (2013), 242--266.


\bibitem{hazi} R. Hazrat,  Graded rings and graded Grothendieck groups, volume 435 of London Mathematical Society Lecture
Note Series. Cambridge University Press, Cambridge, 2016.



\bibitem{rangaswamy} K.M. Rangaswamy, {\it The theory of prime ideals of Leavitt path algebras over arbitrary graphs}, J. Algebra $\mathbf{375}$ (2013), 73--96.

\bibitem{rangaswamy2}  K.M. Rangaswamy, {\it On generators of two-sided ideals of Leavitt path algebras over arbitrary graphs},  Comm.
Algebra, $\mathbf{42(7)}$ (2014), 2859--2868.

\bibitem{tomforde1} M. Tomforde, {\it  Classification of graph algebras: a selective survey}, Operator algebras and applications--the Abel Symposium 2015, 303--325, Abel Symp., 12, Springer, 2017. 
    


\end{thebibliography}
\end{document}